\documentclass[a4paper]{article}
\usepackage{lmodern}
\usepackage[mathjax]{lwarp}
\usepackage{stmaryrd,amsmath,amssymb,amsthm}
\usepackage{fullpage,parskip,enumitem}
\usepackage{hyperref}
\newtheorem{defn}{Definition}

\newtheorem{prop}[defn]{Proposition}
\newtheorem{thm}[defn]{Theorem}

\newtheorem{lem}[defn]{Lemma}

\newtheorem{rem}[defn]{Remark}

\newcommand{\mathmacro}[1]{#1\CustomizeMathJax{#1}}

\mathmacro{
\DeclareMathOperator{\ad}{ad}
\DeclareMathOperator{\Lie}{Lie}

\newcommand{\mb}[1]{\mathbf{#1}}

\newcommand{\Hom}{\mathrm{Hom}}
}

\newcommand\blfootnote[1]{%
    \bgroup
    \renewcommand\thefootnote{\fnsymbol{footnote}}%
    \renewcommand\thempfootnote{\fnsymbol{mpfootnote}}%
    \footnotetext[0]{#1}%
    \egroup
}

\providecommand{\bysame}{\leavevmode\hbox to3em{\hrulefill}\thinspace}
\CustomizeMathJax{\providecommand{\bysame}{\leavevmode\hbox to3em{\hrulefill}\thinspace}}

\numberwithin{equation}{section}
\numberwithin{table}{section}
\numberwithin{defn}{section}

\title{A New Maximal Subgroup of $E_8$ in Characteristic $3$}
\author{David A.\ Craven, David I.\ Stewart and Adam R. Thomas}
\date{5th July, 2021}
\newcommand{\Addresses}{{% additional braces for segregating \footnotesize
  \bigskip
  \footnotesize

  \textsc{David A.\ Craven, School of Mathematics, University of Birmingham, Edgbaston, Birmingham, B15 2TT, United Kingdom}\par\nopagebreak
  \textit{E-mail address}: \texttt{d.a.craven@bham.ac.uk}

  \medskip

  \textsc{David I.\ Stewart, School of Mathematics, Statistics and Physics, Herschel Building, Newcastle University, Newcastle, NE1 7RU, United Kingdom}\par\nopagebreak
  \textit{E-mail address}: \texttt{david.stewart@ncl.ac.uk}

  \medskip

  \textsc{Adam R.\ Thomas, Mathematics Institute, Zeeman Building, University of Warwick, Coventry, CV4 7AL, United Kingdom}\par\nopagebreak
  \textit{E-mail address}: \texttt{Adam.R.Thomas@warwick.ac.uk}

}}

\begin{document}
\maketitle

\begin{abstract}
We prove the existence and uniqueness up to conjugacy of a new maximal subgroup of the algebraic group of type $E_8$ in characteristic $3$. This has type $F_4$, and was missing from previous lists of maximal subgroups produced by Seitz and Liebeck--Seitz. We also prove a result about the finite group $H={}^3\!D_4(2)$, namely that if $H$ embeds in $E_8$ (in any characteristic $p$) and has two composition factors on the adjoint module then $p=3$ and $H$ lies in a conjugate of this new maximal $F_4$ subgroup.
\end{abstract}
\section{Introduction}

\blfootnote{The first author wishes to thank the Royal Society for financial support during the course of
this research.}
The classification of the maximal subgroups of positive dimension of exceptional algebraic groups \cite{liebeckseitz2004} is a cornerstone of group theory. In the course of understanding subgroups of the finite groups $E_8(q)$ in \cite{craven2019un}, the first author ran into a configuration that should not occur according to the tables in \cite{liebeckseitz2004}.

We elicit a previously undiscovered maximal subgroup of type $F_4$ of the algebraic group $E_8$ over an algebraically closed field of characteristic $3$. This discovery corrects the tables in \cite{liebeckseitz2004}, and the original source \cite{seitz1991} on which it depends.

\begin{thm}\label{thm:newf4} Let $\mb G$ be a simple algebraic group of type $E_8$ over an algebraically closed field $k$ of characteristic $3$. Then $\mb G$ contains a unique conjugacy class of simple maximal subgroups of type $F_4$. 

If $\mb X$ is in this class, then the restriction of the adjoint module $L(\mb G)$ to $\mb X$ is isomorphic to $L_{\mb X}(1000)\oplus L_{\mb X}(0010)$, where the first factor is the adjoint module for $\mb X$ of dimension $52$ and the second is a simple module of dimension $196$ for $\mb X$.
\end{thm}

(Here we use the notation $L_{\mb X}(\lambda)$ to refer to the highest-weight module with highest weight $\lambda$ for the algebraic group $\mb X$.)

The classification from \cite{liebeckseitz2004} states that the maximal subgroups of positive dimension are maximal-rank or parabolic subgroups, or one of a short list of reductive subgroups that exist for all but a few small primes, together with $G_2$ inside $F_4$ for $p=7$. This last case arises from a generic embedding of $G_2$ in $E_6$, which falls into $F_4$ on reduction modulo the prime $p=7$ only. This new $F_4$ subgroup of $E_8$ is therefore the only example of a maximal subgroup that exists for a single prime, whose embedding cannot be explained using generic phenomena.

The error in \cite{seitz1991} leading to this new maximal subgroup does not propagate into any other arguments or proofs. The mistake is in calculating the multiplicities of the non-negative $\mathbf T$-weights of the highest-weight modules $L_{\mb X}(0010)$ and $L_{\mb X}(1000)$ in the proof of (15.4), where $\mathbf T$ is a $1$-dimensional torus of a hypothesized maximal subgroup $\mb X=F_4$ of $\mb G = E_8$. The $\mathbf T$-weights on $L(\mb G)$ are determined by a labelling of the Dynkin diagram of $\mb G$ with $0$s and $2$s (a \textit{labelled diagram}). The argument rules out the existence of such a maximal subgroup $\mb X$ with $L(\mb G) \downarrow X = L_{\mb X}(0010) + L_{\mb X}(1000)$ by showing that the $\mathbf T$-weights on $L(\mb G)$ do not come from such a labelled diagram. See \cite[pp.13--16]{seitz1991} for more details about $\mathbf T$, its weights and labelled diagrams. The correct multiplicities for non-negative $\mathbf T$-weights on $L_{\mb X}(0010) + L_{\mb X}(1000)$ are as follows: $$30, 28, 26^2, 24^3, 22^4, 20^5, 18^7, 16^7, 14^9, 12^{10}, 10^{12}, 8^{12}, 6^{14}, 4^{14}, 2^{15}, 0^{16}.$$ There is a labelled diagram which yields these $\mathbf T$-weights, so no contradiction is obtained. 

We note that we have not traced all applications of the theorems in \cite{seitz1991} and \cite{liebeckseitz2004}, for which there are currently 52 and 46 citations respectively according to MathSciNet. Many will be unaffected, or require an extra case to be considered and/or included in final results.   

\medskip

The structure of this note is as follows. Throughout, we let $\mb G$ be a simple algebraic group of type $E_8$ over an algebraically closed field $k$ of characteristic $p$ and let its Lie algebra be denoted by $L(E_8) = \mathfrak{e}_8$. In Section \ref{s:sec2}, we provide a proof of Theorem \ref{thm:newf4}, showing the existence and uniqueness up to conjugacy of a maximal subgroup $\mb X$ of type $F_4$ in $\mb G$ when $p=3$. To do this we first prove that there is an $\mathfrak f_4$ Lie subalgebra of $\mathfrak{e}_8$ that is normalized by the finite group $H = {}^3\!D_4(2)$ and go on to show that there must be a positive-dimensional subgroup $\mb X <  \mb G$ containing $H$, and moreover that $\mb X$ must be a maximal subgroup of type $F_4$. In Section \ref{s:sec3}, a direct construction of $\mb X < \mb G$ is given by providing expressions for the root groups of $\mb X$ in
terms of the root groups of $\mb G$. This can be used to provide an alternative proof of the existence part of Theorem \ref{thm:newf4}. More precisely, we provide a Chevalley basis of the Lie subalgebra $\mathfrak f_4$ from above, and exponentiating this provides another alternative proof of the existence part of Theorem \ref{thm:newf4}. 

%Throughout, we let $\mb G$ be a simple algebraic group of type $E_8$ over an algebraically closed field $k$ of characteristic $3$. The existence and uniqueness up to conjugacy of a maximal Lie subalgebra $\mathfrak f_4\subset \mathfrak e_8$ is first established. From this we are able to write down the root elements of $\mathfrak f_4$; exponentiation gives expressions for the root groups of $\mb X$ in terms of the root groups of $\mb G$, providing an explicit construction of $\mb X$.

In the final section we determine various results providing extra details on this new class of maximal subgroups. For each unipotent class in $\mb X$ we determine the corresponding unipotent class in $\mb G$ that contains it, and we do the same for nilpotent orbits of the corresponding Lie algebras. We also consider the maximal connected subgroups of $\mb X$. The maximal parabolic subgroups of $\mb X$ will be contained in parabolic subgroups of $\mb G$ by the Borel--Tits Theorem. There are four classes of reductive maximal connected subgroups of $\mb X$ when $p=3$ with types $B_4, A_1 C_3, A_1 G_2, A_2A_2$. We show that all of these classes are contained in other maximal connected subgroups of $\mb G$ and we specify such an overgroup. Moreover, we determine that the first three classes are $\mb G$-irreducible but the last class is not. (A subgroup is $\mb H$-irreducible for some connected reductive algebraic group $\mb H$ if it is not contained in any proper parabolic subgroup of $\mb H$.)

In establishing the existence of $\mb X$, we prove the following extra result, of use in the project to classify maximal subgroups of the finite exceptional groups of Lie type.

\begin{prop} Let $H$ be the group ${}^3\!D_4(2)$, let $p$ be a prime, and suppose that $H$ embeds in the algebraic group $E_8$ in characteristic $p$. If the composition factors of the action of $H$ on the adjoint module $L(E_8)$ have dimensions $52$ and $196$, then $p=3$ and $H$ is contained in a maximal subgroup $\mb X$ of type $F_4$; furthermore, $H$ and $\mb X$ stabilize the same subspaces of $L(E_8)$.
\end{prop}

\section{From the Thompson group to \texorpdfstring{$F_4$}{F4}} \label{s:sec2}

Let $p$ be an odd prime and let $k$ be an algebraically closed field of characteristic $p$. From the end of Lemma \ref{l:f4subalgebra} onwards we will assume that $p=3$.

One path to a construction of the $F_4$ subgroup of $E_8$ starts with the Thompson group (acting irreducibly on $L(E_8)$ \cite{smith1976}), which contains a copy of $H\cong {}^3\!D_4(2)$, acting on $L(E_8)$ with composition factors of dimensions $52$ and $196$ (using the trace information in \cite[p.176]{atlas} and \cite[p.251]{abc}). In fact, we will show that every $H$-invariant alternating bilinear form on the $248$-dimensional module is invariant under a suitable copy of $F_4\leq \mathrm{GL}_{248}(k)$, where $k$ is algebraically closed and of characteristic $3$.

We cannot quite show this without a computer. Splitting $L(E_8)$ up as the sum of $52$ and $196$ fragments the space of alternating forms into six components. For five of these six we can show that the $H$-invariant maps are $F_4$-invariant, but for the sixth we cannot do so without a computer. With a computer we can check that this sixth component is at least $F_4(9)$-invariant, and thus every subgroup $H$ of $\mb G$ is contained in a copy of $F_4(9)$. But $F_4(9)$ contains elements of order $6562$, and thus there is an $F_4$ subgroup of $\mb G$ containing it, via \cite[Proposition~2]{liebeckseitz1998} and Lemma \ref{l:itisF4} below.

We then show that this $F_4$ subgroup is unique up to $\mb G$-conjugacy, obtaining as a by-product that $H$ is unique up to $\mb G$-conjugacy.

\medskip

We start with a copy $J$ of the Thompson sporadic simple group. This has a $248$-dimensional self-dual simple module $M$ over $\mathbb C$ (it is a minimal faithful representation), and it remains simple upon reduction modulo all primes. From \cite[p.176]{atlas}, we see that there are elements of order $9$ with Brauer character value $5$ on $M$. The only integers that are the traces of semisimple elements of order $9$ in $\mb G$ (on the adjoint module) are $-1$, $2$, $8$ and $29$. (These were computed using the algorithm in \cite[Appendix]{litterick}, see Table \ref{t:trace9}.)
\begin{table}
\begin{center}
\begin{tabular}{cccc}
\hline Trace of $x^3$ & Trace of $x$ & Trace of $x^3$ & Trace of $x$
\\\hline 
$77$ & $56\lambda_1+\lambda_2+134$& $-4$ & $ 56 \lambda_1+ 28 \lambda_2+ 56$
\\ & $26\lambda_1+25\lambda_2+53$& & $ 35 \lambda_1+ 28 \lambda_2+ 38$
\\ & $20\lambda_1+\lambda_2+35$& & $ 26 \lambda_1+ 16 \lambda_2+ 20$
\\ & $8\lambda_1+7\lambda_2+8$& & $ 11 \lambda_1+ 10 \lambda_2+ 20$
\\ & $2\lambda_1+\lambda_2+8$& & $ 8 \lambda_1+ \lambda_2+ 11$
\\ & $2\lambda_1+\lambda_2-1$ & & $ 2 \lambda_1+ \lambda_2+ 2$
\\ \hline$5$& $ 54 \lambda_1+ 27 \lambda_2+ 80$& $24$ &$64 \lambda_1+ 14 \lambda_2+ 92$
\\ & $ 45 \lambda_1+ 27 \lambda_2+ 38$&&$43 \lambda_1+ 32 \lambda_2+ 56$
\\ & $ 24 \lambda_1+ 21 \lambda_2+ 32$&&$34 \lambda_1+ 20 \lambda_2+ 29$
\\ & $ 18 \lambda_1+ 9 \lambda_2+ 17$&&$16 \lambda_1+ 8 \lambda_2+ 29$
\\ & $ 9 \lambda_1+ 9 \lambda_2+ 11$&&$13 \lambda_1+ 20 \lambda_2+ 29$
\\ & $ 6 \lambda_1+ 3 \lambda_2+ 5$&&$13 \lambda_1+ 8 \lambda_2+ 11$
\\ & $29$&&$4 \lambda_1+ 5 \lambda_2+ 11$
\\ & $8$&&$4 \lambda_1+ 2 \lambda_2+ 2$
\\ & $2$&&$\lambda_1+ 2 \lambda_2+ 2$
\\ & $-1$&&
\\\hline
\end{tabular}\end{center}
\caption{List of traces of semisimple elements $x$ of order $9$ in $E_8$ on $L(E_8)$. Here, $\omega_9$ is a primitive $9$th root of unity, $\lambda_1=\omega_9+\omega_9^{-1}$ and $\lambda_2=\omega_9^2+\omega_9^{-2}$. Traces are given up to algebraic conjugacy.}
\label{t:trace9}
%\begin{center}
%\begin{tabular}{cc}
%\hline Trace of $x^3$ & Trace of $x$
%\\\hline 
%$77$ & $56\lambda_1+\lambda_2+134$
%\\ & $26\lambda_1+25\lambda_2+53$
%\\ & $20\lambda_1+\lambda_2+35$
%\\ & $8\lambda_1+7\lambda_2+8$
%\\ & $2\lambda_1+\lambda_2+8$
%\\ & $2\lambda_1+\lambda_2-1$ 
%\\ \hline $24$ &$64 \lambda_1+ 14 \lambda_2+ 92$
%\\&$43 \lambda_1+ 32 \lambda_2+ 56$
%\\&$34 \lambda_1+ 20 \lambda_2+ 29$
%\\&$16 \lambda_1+ 8 \lambda_2+ 29$
%\\&$13 \lambda_1+ 20 \lambda_2+ 29$
%\\&$13 \lambda_1+ 8 \lambda_2+ 11$
%\\&$4 \lambda_1+ 5 \lambda_2+ 11$
%\\&$4 \lambda_1+ 2 \lambda_2+ 2$
%\\&$\lambda_1+ 2 \lambda_2+ 2$
%\\ \hline$5$& $ 54 \lambda_1+ 27 \lambda_2+ 80$
%\\ & $ 45 \lambda_1+ 27 \lambda_2+ 38$
%\\ & $ 24 \lambda_1+ 21 \lambda_2+ 32$
%\\ & $ 18 \lambda_1+ 9 \lambda_2+ 17$
%\\ & $ 9 \lambda_1+ 9 \lambda_2+ 11$
%\\ & $ 6 \lambda_1+ 3 \lambda_2+ 5$
%\\ & $29$
%\\ & $8$
%\\ & $2$
%\\ & $-1$
%\\ \hline $-4$ & $ 56 \lambda_1+ 28 \lambda_2+ 56$
%\\ & $ 35 \lambda_1+ 28 \lambda_2+ 38$
%\\ & $ 26 \lambda_1+ 16 \lambda_2+ 20$
%\\ & $ 11 \lambda_1+ 10 \lambda_2+ 20$
%\\ & $ 8 \lambda_1+ \lambda_2+ 11$
%\\ & $ 2 \lambda_1+ \lambda_2+ 2$
%\\ \hline 
%\end{tabular}\end{center}
\end{table} Thus these elements cannot be semisimple, and in particular, $p\mid 9$. Thus we see that if $J$ embeds in the algebraic group $\mb G$ in any characteristic $p$, then $p=3$. It is a famous result \cite{smith1976} that $J$ does indeed embed in $E_8(3)$, and is unique up to conjugacy.

It is well known that $J$ contains a subgroup $H$ isomorphic to ${}^3\!D_4(2)$. From \cite[p.90]{atlas} and \cite[pp.251--253]{abc}, we see that in characteristic not $2$, the restriction of $M$ to $H$ is the direct sum of a $52$-dimensional simple module $M_1$ and a $196$-dimensional simple module $M_2$ (the sum is direct since $M$ is self-dual). However, all elements of order $9$ in $H$ act on $M_1\oplus M_2$ with trace $2$, so we cannot use the previous method to show that $H$ cannot embed in $\mb G$ in characteristic $p\neq 2,3$ acting on $L(\mb G)$ as $M_1\oplus M_2$.

\begin{lem} \label{l:f4subalgebra} Let $p$ be an odd prime, let $H$ denote the group ${}^3\!D_4(2)$, and suppose that $H$ embeds in the algebraic group $\mb G$ over an algebraically closed field $k$ of characteristic $p$, acting on the adjoint module with composition factors of dimensions $52$ and $196$. Then the $52$-dimensional submodule carries the structure of a Lie algebra of type $F_4$, and in addition $p=3$. Furthermore, such an $\mathfrak f_4$-subalgebra of $\mathfrak{e}_8$ does exist for $p=3$.
\end{lem}
\begin{proof} Let $M_1$ denote the $52$-dimensional $kH$-submodule of the adjoint module for $\mb G$ and $M_2$ the $196$-dimensional submodule. Note that $|H|=2^{12}\cdot 3^4\cdot 7^2\cdot13$, so either $p\nmid |H|$ and we are essentially in characteristic $0$, or $p=7,13$, or $p=3$.

From \cite[Table 1.1]{craven2020un}, there is a unique conjugacy class of subgroups $H$ in $F_4$ for any odd characteristic $p$, acting irreducibly on the minimal and adjoint modules. In particular, we see from \cite[Section 4.3.4]{craven2020un} that $\Hom_{kH}(\Lambda^2(M_1),M_1)$ is $1$-dimensional for all odd primes $p$. From an ordinary character calculation, we see that
\[ \Lambda^2(\chi_{52})=\chi_{52}+\chi_{1274},\]
where $\chi_i$ is the irreducible character of $H$ of degree $i$. For $p=7,13$, the reduction modulo $p$ of $\chi_{1274}$ is irreducible (see \cite[pp.252--253]{abc}), hence the reduction modulo $p$ is irreducible modulo $p$ for all $p>3$ (as $p=7,13$ are the only primes greater than $3$ dividing $|H|$). Thus for $p>3$, if $H$ embeds in $\mb G$ with the claimed composition factors then the $52$-dimensional summand is a Lie subalgebra of $L(\mb G)$.

On the other hand, if $p=3$ then $\chi_{1274}$ has Brauer character constituents of degrees $52$ and $1222$ from \cite[p.251]{abc}. Since $\Hom_{kH}(\Lambda^2(M_1),M_1)=k$, we see that the exterior square is uniserial, with layers of dimensions $52$, $1222$ and $52$. Moreover,
\[\Hom_{kH}(\Lambda^2(M_1),M_2)=0.\]
Thus again $M_1$ forms an $H$-invariant subalgebra. Thus for all odd primes $p$, $M_1$ is an $H$-invariant Lie subalgebra of $L(\mb G)$ of dimension $52$.

Moreover, $M_1$ must be non-abelian for all $p$, since $\mathfrak{e}_8$ contains no abelian subspace of dimension $52$ by \cite[Proposition~2.3]{craven2021un}. Furthermore, as $M_1$ is irreducible for $H$, the restriction of the Lie bracket to $M_1$ furnishes it with the structure of a semisimple Lie algebra. Since $\Hom_{kH}(\Lambda^2(M_1),M_1)=k$, there is at most one isomorphism class of such, but as the algebraic $k$-group $F_4$ does contain a subgroup isomorphic to $H$, acting as $\chi_{52}$ on its adjoint module, it follows that $M_1\cong\mathfrak{f}_4$. In particular, this means that the $\mathfrak f_4$ Lie algebra must have a simple module of dimension $196$, as it acts $H$-equivariantly on $L(\mb G)$.

Such a simple module must be restricted: if not the $p$-closure $L_p$ of the image $L$ of $\mathfrak f_4$ in $\mathfrak e_8$ will contain a non-trivial centre \cite[2.5.8(2)]{SF88}; but $L$ has no $1$-dimensional submodules on $\mathfrak e_8$. By Curtis's theorem, $M_2$ arises by differentiation of a restricted representation for the algebraic group $F_4$, whence $p=3$, from the tables in \cite{luebeck2001}.

The embeddings of ${}^3\!D_4(2)$ and the $\mathfrak{f}_4$-subalgebra do exist for $p=3$ via the Thompson group, as seen above.
\end{proof}

For the rest of this paper we therefore assume that $p=3$.

We will prove that the $\mathfrak f_4$-subalgebra is the Lie algebra of an $F_4$ algebraic subgroup of $\mb G$. To do so, we will actually prove that every $H$-invariant alternating product on the $kH$-module $M=M_1\oplus M_2$ for $p=3$ is also $F_4$-invariant, for the unique $F_4\leq \mathrm{GL}_{52}(k)$ containing $H$. To do so, we need to understand the space
\[ \Hom_{kH}(\Lambda^2(M),M)\]
of alternating products on $M$. Using $M=M_1\oplus M_2$, and the formula
\[ \Lambda^2(A\oplus B)\cong \Lambda^2(A)\oplus \Lambda^2(B)\oplus A\otimes B,\]
we split the space of products up into six components. The next result gives the dimensions of these components.

\begin{prop} We have
\[ \Hom_{kH}(\Lambda^2(M_1),M_1)=k,\qquad \Hom_{kH}(\Lambda^2(M_1),M_2)=0,\]
\[ \Hom_{kH}(\Lambda^2(M_2),M_1)=k,\qquad \Hom_{kH}(\Lambda^2(M_2),M_2)=k,\]
\[ \Hom_{kH}(M_1\otimes M_2,M_1)=0,\qquad \Hom_{kH}(M_1\otimes M_2,M_2)=k.\]
\end{prop}
\begin{proof} One may use a computer to check these with ease. Some may be checked easily by hand as well, using the ordinary character table and the $3$-decomposition matrix for $H$. For example, using those two tables, $M_1\otimes M_2$ does not possess a composition factor $M_1$, and thus \[ \Hom_{kH}(M_1\otimes M_2,M_1)=\Hom_{kH}(\Lambda^2(M_1),M_2)=0.\]
The statement that $\Hom_{kH}(\Lambda^2(M_1),M_1)=k$ appears in \cite[Section 4.3.4]{craven2020un} (where it is proved by computer).
\end{proof}

At least the existence, if not the uniqueness, of two of the three remaining non-zero maps is clear from the fact that $H$ embeds in $E_8(3)$ with representation $M_1\oplus M_2$. If $\Hom_{kH}(M_1\otimes M_2,M_2)=0$ then $M_1$ would be an ideal of the Lie algebra (as $\Hom_{kH}(\Lambda^2(M_1),M_2)=0$), which is not possible. A character calculation shows that $S^2(M_2)$ does not have a composition factor $M_1$, so
\[ \Hom_{kH}(\Lambda^2(M_2),M_1)=\Hom_{kH}(M_1\otimes M_2,M_2).\]
%Finally, to avoid $M_2$ being an ideal we need $\Hom_{kH}(\Lambda^2(M_2),M_1)$ to be non-zero.

It is only $\Hom_{kH}(\Lambda^2(M_2),M_2)$ that cannot easily be seen. Indeed, this space will cause us a problem later on.

\bigskip

We now prove that five of the six Hom-spaces extend to the algebraic group $\mb X=F_4$, with only $\Hom_{kH}(\Lambda^2(M_2),M_2)$ missing. If one is happy to use a computer for all of this, one simply checks that all $H$-invariant maps are $F_4(3)$- and even $F_4(9)$-invariant, and thus one does not need to prove the next proposition.

\begin{prop} Let $\mb X$ be an algebraic $k$-group of type $F_4$. We have
\[ \Hom_{\mb X}(\Lambda^2(L(1000)),L(1000))=k,\]
\[ \Hom_{\mb X}(\Lambda^2(L(1000)),L(0010))=\Hom_{\mb X}(L(1000)\otimes L(0010),L(1000))=0, \text{ and }\]
\[ \Hom_{\mb X}(\Lambda^2(L(0010)),L(1000))=\Hom_{\mb X}(L(1000)\otimes L(0010),L(0010))=k.\]
\end{prop}
\begin{proof} Note that each of these spaces must have dimension at most the dimension of the corresponding space for $H$. This yields the two $0$-dimensional spaces, and that the other spaces have dimension at most $1$. The first statement holds because $\mb X$ is an algebraic group and its adjoint module is $L(1000)$, thus the space is non-zero.

For the last statement, since $S^2(M_2)$ has no composition factor isomorphic to $M_1$, certainly $S^2(L(0010))$ has no composition factor isomorphic to $L(1000)$. Thus the two Hom-spaces are isomorphic, so it remains to find a non-zero map in the latter space.

Using the Brauer character table of $H$ \cite[p.251]{abc} (or preferably, a computer), the composition factors of the $kH$-module $M_1\otimes M_2$ are of dimensions
\[ 25,196,196,441,1963,2457,2457,2457.\]
The highest-weight module $L(1010)$, which must appear as a composition factor in $L(1000)\otimes L(0010)$, has dimension $7371$ (see \cite[Appendix A.50]{luebeck2001}), and must restrict to $kH$ to be the sum of the three (non-isomorphic) modules of dimension $2457$, as no other combination of dimensions works. The rest of the composition factors, in total, have dimension $2821$, so there must be an $\mathbf{X}$-composition factor of dimension between $1963$ and $2821$. Consulting \cite[Appendix A.50]{luebeck2001}, we find exactly one such module: $L(0011)$ of dimension $2404=1963+441$. The remaining $kH$-modules, $25$, $196$ and $196$, must be the other composition factors for $\mathbf{X}$, because $\mathbf{X}$ has no simple modules of dimension $25+196$, $196+196$, or $25+196+196$.

Thus the composition factors of $L(1000)\otimes L(0010)$ have dimensions $25$, $196$, $196$, $2404$ and $7371$. Since $L(0010)$ is the unique module to appear more than once, and the tensor product is self-dual, $L(0010)$ must be a submodule, and the maps in $\Hom_{kH}(M_1\otimes M_2,M_2)$ extend to $\mb X$.
\end{proof}

The last remaining Hom-space to check is $\Hom_{\mb X}(\Lambda^2(L(0010)),L(0010))$. This seems difficult to do by hand, and we resort to a computer. There are two ways to proceed. The first is to prove that there is an $F_4(3)$-invariant map in the space (this takes a couple of minutes), and thus the group $H$ is contained in a copy of $F_4(3)$ in $\mb G$. We then apply \cite[Proposition 6.8]{craven2019un}, which states that $F_4(3)$ is contained in a positive-dimensional subgroup stabilizing the same subspaces of $L(\mb G)$, which are $M_1$ and $ M_2$. This must be a copy of $F_4$ (as it stabilizes an $\mathfrak f_4$-Lie subalgebra), and we are done. Alternatively, we prove the same statement for $F_4(9)$ (which contains elements of order $9^4+1=6562$, this takes about half an hour on one of the first author's computers) and then apply \cite[Proposition 2]{liebeckseitz1998}, which yields the same positive-dimensional subgroup.

We must also show that the subgroup $\mb X$ is actually $F_4$. This is easy, and we can do it quite generally.

\begin{lem}\label{l:itisF4} Let $\mb G$ be of type $E_8$ over $k$, and let $\mb X$ be a closed, positive-dimensional subgroup of $\mb G$. If $\mb X$ acts on $L(\mb G)$ with composition factors of dimensions $52$ and $196$ then $\mb X$ is simple of type $F_4$. Furthermore, $\mb X$ is maximal in $\mb G$.
\end{lem}
\begin{proof} Since $\mb X$ has no trivial composition factor on $L(\mb G)$ it cannot lie in a parabolic subgroup, hence must be reductive. It also cannot centralize any semisimple element, hence must be semisimple. Since $L(\mb X)$ is a submodule of $L(\mb G)$, it has dimension either $52$ or $196$ and is simple. There is no simple algebraic group of rank at most $8$ and dimension $196$, so $\mb X$ has dimension $52$, and must therefore be $F_4$. Since $F_4$ has no outer automorphisms, $N_{\mb G}(\mb X)=\mb X$ (as $\mb X$ has trivial centralizer). Since any closed, positive-dimensional proper subgroup of $\mb G$ cannot act irreducibly on $L(\mb G)$, any overgroup of $\mb X$ also acts with composition factors $52$ and $196$, hence is $F_4$ by the above proof. Thus $\mb X$ is maximal, as claimed.
\end{proof}

Thus we obtain the following.

\begin{prop}\label{uniquef4} Let $p=3$ and let $H$ be a subgroup ${}^3\!D_4(2)$ of $\mb G$, acting on $L(\mb G)$ with composition factors of dimensions $52$ and $196$. Then $H$ is contained in a positive-dimensional subgroup of type $F_4$, stabilizing exactly the same subspaces of $L(\mb G)$ that are stabilized by $H$.
\end{prop}

It suffices to ascertain the uniqueness up to conjugacy of the subgroup of type $F_4$, and as a by-product we also obtain uniqueness of $H$ up to conjugacy.

Let $t$ be an involution in $\mb X=F_4$ with centralizer $\mathbf{B}$ of type $B_4$. The trace of $t$ on $L(\mb G)$ is $-8$ (see \cite[Table 4]{liebeckseitz1999}), and so the centralizer of $t$ in $\mb G$ is $D_8$, which acts with composition factors $L(\lambda_2)$ and $L(\lambda_7)$, of dimensions $120$ and $128$ respectively. The restriction of the $k\mb X$-module $L_\mathbf{X}(1000)$ to $\mathbf{B}$ is $L_\mathbf{B}(0001)\oplus L_\mathbf{B}(0100)$% (to avoid confusion we use $L_\mathbf{X}(-)$ and $L_\mathbf{B}(-)$ to denote the highest-weight modules for $\mathbf{X}$ and $\mathbf{B}$ respectively)
, with dimensions $16$ and $36$ respectively. The restriction of $L_\mathbf{X}(0010)$ to $\mathbf{B}$ is $L_\mathbf{B}(0010)\oplus L_\mathbf{B}(1001)$, of dimensions $84$ and $112$ respectively. (This can be checked using weights or quickly on a computer for $F_4(3)$.) From \cite[Table~60]{thomas2017un} we see that $\mathbf{B}$ is subgroup $E_8(\#45)$ and is unique up to conjugacy. But clearly there is a unique way to assemble the numbers $16$, $36$, $84$ and $112$ to make $52$ and $196$. Thus given any subgroup of type $B_4$ there exists at most one $F_4$ containing it, which must stabilize the submodule $L_\mathbf{B}(0001)\oplus L_\mathbf{B}(0100)$. Thus we obtain the result that $\mb X$ is unique up to conjugacy.

This completes the proof of uniqueness of $\mb X$, and thus Theorem \ref{thm:newf4} is proved.

\section{An explicit construction of the maximal subgroup \texorpdfstring{$F_4$}{F4}} \label{s:sec3}
Recall that $k$ is algebraically closed and of characteristic $3$. For their application to future explicit computations, we give expressions for the root elements $x_{\pm \beta_i}(t)$ for $t \in k$ and $\beta_i$ a root of the maximal subgroup $F_4$ as products of root elements of $\mb G$; see \cite[\S4.4]{Car89} for notation. We also provide the elements $h_{\gamma_i}(t)$ for $\gamma_i$ a simple root of $F_4$ and $t \in k^*$, written in terms of the $h_{\alpha_i}(t)$ elements of $\mb G$. As the factors of $x_{\beta_i}(t) = \prod_j x_{\alpha_j}(c_j t)$ commute we get easily that $x_{-\beta_i}(t) = \prod_j x_{-\alpha_j}(c_j t)$.  For this reason we only list the root group elements for positive roots. Note that since the coefficients $c_j$ that we exhibit are elements of $\text{GF}(3)$, it follows that the subgroup $\mathbf Y$ we produce is defined over $\text{GF}(3)$. 

\begin{prop} \label{p:almostpres}
The elements $x_{\pm \beta_i}(t) = \prod_j x_{{\alpha_j}}(c_j t)$ for $t \in k$ generate a maximal subgroup $\mb X$ of type $F_4$ in $\mb G$. They are the root elements of $F_4$ with respect to the maximal torus generated by $h_{\gamma_i}(t) \prod_j x_{{\alpha_j}}(d_j t)$ for $t \in k^*$. Furthermore, the elements $e_{\beta_i}, h_{\gamma_i}$ form a Chevalley basis for $\mathfrak h = \Lie(F_4)$ where $e_{\beta_i} = \sum c_je_{\alpha_j}$,  $h_{\gamma_i} = \sum d_je_{\alpha_j}$. 
\end{prop}

$x_{1000}(t) = x_{00010000}(t) x_{00000100}(t)$

$x_{0100}(t) = x_{00100000}(-t) x_{00000010}(t)$

$x_{0010}(t) = x_{10000000}(-t) x_{00000001}(t) x_{00011000}(-t) x_{00001100}(-t)$

$x_{0001}(t) = x_{11110000}(t) x_{01121000}(-t) x_{01111100}(-t) x_{01011110}(t)$

$x_{1100}(t) = x_{00110000}(-t) x_{00000110}(-t)$ 

$x_{0110}(t) = x_{10100000}(t) x_{00000011}(-t) x_{00111000}(-t) x_{00001110}(-t) $ 

$x_{0011}(t) = x_{11121000}(t)x_{11111100}(t)x_{01011111}(t)x_{01122100}(-t)$ 

$x_{1110}(t) = x_{10110000}(t)x_{00000111}(t)x_{00111100}(-t)x_{00011110}(t)$ 

$x_{0120}(t) = x_{10111000}(-t)x_{00001111}(-t) $

$x_{0111}(t) = x_{11221000}(t)x_{11111110}(t)x_{01111111}(t)x_{01122110}(-t)$

$x_{1120}(t) = x_{10111100}(-t)x_{00011111}(t)$ 

$x_{1111}(t) = x_{11221100}(t)x_{11121110}(-t)x_{01121111}(-t)x_{01122210}(t)$ 

$x_{0121}(t) = x_{11111111}(-t) x_{11222100}(-t) x_{11122110}(t) x_{01122111}(t) $

$x_{1220}(t) = x_{10111110}(t) x_{00111111}(-t)$ 

$x_{1121}(t) = x_{11121111}(t) x_{11232100}(t) x_{11122210}(-t) x_{01122211}(-t)$ 

$x_{0122}(t) = x_{12232111}(-t) x_{12233210}(t)$ 

$x_{1221}(t) = x_{11221111}(-t) x_{11232110}(-t) x_{11222210}(t) x_{01122221}(-t)$

$x_{1122}(t) = x_{12232211}(t) x_{12243210}(-t)$ 

$x_{1231}(t) = x_{11232111}(t) x_{11222211}(-t) x_{11122221}(-t) x_{11233210}(t)$ 

$x_{1222}(t) = x_{12232221}(t) x_{12343210}(t)$ 

$x_{1232}(t) = x_{22343210}(t) x_{12343211}(t) x_{12243221}(t) x_{12233321}(-t)$

$x_{1242}(t) = x_{22343211}(t) x_{12244321}(t)$ 

$x_{1342}(t) = x_{22343221}(-t) x_{12344321}(t)$ 

$x_{2342}(t) = x_{22343321}(t) x_{12354321}(-t)$

$h_{1000}(t) = h_{00010000}(t) h_{00000100}(t)$

$h_{0100}(t) = h_{00100000}(t) h_{00000010}(t)$

$h_{0010}(t) = h_{10000000}(t) h_{00010000}(t) h_{00001000}(t^2) h_{00000100}(t) h_{00000001}(t) $

$h_{0001}(t) = h_{1000000}(t) h_{01000000}(t^4) h_{00100000}(t^3) h_{00010000}(t^5) h_{00001000}(t^3) h_{00000100}(t^2) h_{00000010}(t)$

\begin{proof}
It can be checked by hand that these elements satisfy Steinberg's relations (see \cite[Theorem~12.1.1]{Car89}) for a suitable choice of constants for Chevalley's commutator relations and that the corresponding elements of $\mathfrak{e}_8$ form a Chevalley basis for $\text{Lie}(\mb X)$, a Lie algebra of type $F_4$. Another check, which theoretically could be done by hand, shows that $\mb X$ acts on $\text{Lie}(\mb G)$ as $L_{\mb X}(1000)\oplus L_{\mb X}(0010)$, and so by Lemma \ref{l:itisF4} we see that $\mb X$ is a maximal subgroup of $\mb G$. 
\end{proof}

\begin{rem}
The proof of Proposition \ref{p:almostpres} provides another proof for the existence of a maximal subgroup of type $F_4$ in $\mb G$, which is at least in theory computer-free.  
\end{rem}

\begin{rem} 
We note that the $24$-dimensional unipotent subgroup generated by the positive root elements of $F_4$ is contained in the $120$-dimensional unipotent subgroup generated by the positive roots of $\mb G$. The presentation alluded to in Proposition \ref{p:almostpres} has the slightly unfortunate property that the constants in Chevalley's commutator relations for this subgroup $F_4$ are not the same as those used in Magma, which are a somewhat standard choice. However, one can if one wishes rectify this by choosing a different basis: let $\tilde{x}_{\alpha_1}(t) = x_{-0100}(t)$, $\tilde{x}_{\alpha_2}(t) = x_{-1242}(t)$, $\tilde{x}_{\alpha_3}(t) = x_{1232}(t)$, $\tilde{x}_{\alpha_4}(t) = x_{-0001}(t)$. Then we generate the same maximal subgroup $F_4$ but this time the constants in the commutator relations do agree with those in Magma.\end{rem}

We provide a brief explanation on how we found the generators and Chevalley basis in Proposition \ref{p:almostpres}. We start with the $248$-dimensional module $M$ with summands of
dimensions $52$ and $196$ for $H\cong {}^3D_4(2)$ (defined at the start of Section \ref{s:sec2}). We use the space
$\Hom_{kH}(\Lambda^2(M),M)$, as used in \cite{craven2021un} (described
there as the `Lie product method'), to construct an explicit
$H$-invariant Lie product on $M$ that turns $M$ into a copy of
$\mathfrak{e}_8$. This gives us explicit structure constants. The module $M$ splits as the sum of $52$- and $196$-dimensional $H$-stable submodules. As explained in Lemma \ref{l:f4subalgebra}, the first of these subspaces is forced to be a subalgebra $\mathfrak{h}$ of $\mathfrak{e}_8$ isomorphic to $\mathfrak{f}_4$ and so Magma could write down a basis for it in terms of a basis of $\mathfrak{e}_8$. However, this process left us with basis elements for $\mathfrak h$ with around $120$ non-zero coefficients in terms of a basis of $\mathfrak{e}_8$. 

We found that four of the basis vectors for $\mathfrak{h}$ were toral (meaning $x^{[p]} = x$ \cite{SF88}, which implies $x$ is semisimple) and commuted with each other, thus spanning a maximal toral subalgebra $\mathfrak{t}$. We then searched for a $\mb G$-conjugate of $\mathfrak{h}$ such that the corresponding conjugate of $\mathfrak{t}$ was contained in the standard toral subalgebra of $\mathfrak{e}_8$. To do this we used the inbuilt \texttt{InnerAutomorphism} function in Magma to construct the automorphisms of $\mathfrak{e}_8$ corresponding to $x_{\gamma}(\pm1)$ for all roots $\gamma$ in the root system of $\mb G$, yielding 480 possible conjugations.

Our strategy was to implement a naive hill climb for the first basis element $t_1$ of $\mathfrak{t}$. Indeed, we searched through all 480 possible conjugating elements and selected the one that yielded the largest number of zero coefficients when expressing $t_1^g$ in terms of the basis of $\mathfrak{e}_8$. We remembered the elements we used at each step. This meant that when we could no longer increase the number of zero coefficients we could trace our steps back and take the next best conjugating element and continue the process. This lead to a significant increase in the number of zero coefficients but nowhere near the $240$ we needed.  

We then slightly upgraded our hill climb algorithm to include using a random conjugating element at fixed intervals. Every $100$ steps we chose a random conjugating element and used this, regardless of what it did to the number of zero coefficients. This method was not optimized; it could be that a better choice would have been every five steps, or 500 steps. But this hill climb was enough for us; it quickly led us to a conjugating element $g_0$ which took $t_1$ into the standard toral subalgebra of $\mathfrak{e}_8$. At this point, the remaining three basis elements $t_2, t_3$ and $t_4$ were not sent by $g_0$ to something in the standard toral subalgebra of $\mathfrak{e}_8$, but they had significantly fewer non-zero coefficients. We then repeated the algorithm looking to increase the total number of zero coefficients in $t_1^g, \ldots, t_4^g$ and this quickly converged, yielding a conjugating element $g_1$ which sent $t_1, \ldots, t_4$ to the four toral elements corresponding to the generators of the maximal torus given in Proposition \ref{p:almostpres}. From the toral subalgebra of $\mathfrak{h}^{g_1}$ it was then routine to take a Cartan decomposition and find the corresponding root elements. 

It turned out that the root elements of $\mathfrak h$ were expressed as a sum of commuting root elements of $\mathfrak{e}_8$; in fact long root elements of $\mathfrak h$ were of type $2A_1$ in $\mathfrak{e}_8$, whereas short root elements were of type $4A_1$. For pairwise commuting root elements $e_{\alpha_1},\dots,e_{\alpha_k}$ of $\mathfrak{e}_8$, the operators $\ad e_{\alpha_i}$ pairwise commute and so one has
\[\exp(t_1\ad e_{\alpha_1}+\dots + t_m\ad e_{\alpha_m})=\exp(t_1e_{\alpha_1})\ldots\exp(t_m e_{\alpha_m}).\]
Thus if $e_{\beta_i}=\sum t_je_{\alpha_j}\in\mathfrak{h}$ is of this form, then evidently the left-hand side of the displayed equation normalizes $\mathfrak{h}$ in the group $\mathrm{GL}_{248}(k)$, and the right-hand side belongs to $E_8(k)$, hence for $t\in k$, the elements $x_{\beta_i}(t)$ generate a connected smooth subgroup $\mathbf Y$ of $E_8$ for which $\mathfrak{h}\subseteq\Lie(\mathbf Y)$. As the restriction to $\mathfrak h$ of the adjoint module $\mathfrak e_8$ is the direct sum of $\mathfrak{h}$  and a simple module, it is maximal, which forces $\Lie(\mathbf Y)=\mathfrak{h}$. The only connected smooth affine $k$-group whose Lie algebra is a simple Lie algebra of type $F_4$ is a group of type $F_4$ itself and using the maximality of $\mathfrak{f}_4$ we conclude that $\mathbf Y$ must be a maximal connected subgroup. This yields yet another proof of the existence a maximal subgroup of type $F_4$ in $\mb G$. 

\section{Consequences}

As before, $k$ is algebraically closed and of characteristic $3$. We extend the results of \cite{lawther2009} to this new maximal subgroup, determining which unipotent classes of $\mb G$ meet the new maximal subgroup $\mb X$ of type $F_4$ non-trivially.

\begin{prop} If $u$ is a unipotent element of $\mb X < \mb G$, then the class of $u$ in $\mb X$ and $\mb G$ is given in Table \ref{tab:classfusionalg}.
\begin{table}
\begin{center}\begin{tabular}{cc}
\hline Class in $F_4$ & Class in $E_8$
\\ \hline $A_1$&$2A_1$
\\ $\tilde A_1$ & $4A_1$
\\ $A_1+\tilde A_1$ & $A_2+2A_1$
\\ $A_2$ & $2A_2$
\\ $\tilde A_2$ & $2A_2$
\\ $A_2+\tilde A_1$ & $2A_2+A_1$
\\ $\tilde A_2+A_1$ & $2A_2+2A_1$
\\\hline $B_2$ & $2A_3$
\\ $C_3(a_1)$ & $A_4+2A_1$
\\ $F_4(a_3)$ & $A_4+A_2$
\\ $B_3$ & $A_6$
\\ $C_3$ & $D_6(a_1)$
\\ $F_4(a_2)$ & $D_5+A_2$
\\ $F_4(a_1)$ & $E_8(b_6)$
\\\hline $F_4$ & $E_8(b_4)$
\\ \hline
\end{tabular}
\end{center}
\caption{Fusion of unipotent classes of the maximal subgroup $\mb X = F_4$ into $\mb G = E_8$. (Horizontal lines separate elements of different orders.)}
\label{tab:classfusionalg}
\end{table}
\end{prop}

\begin{proof} The proof is a fast computer check. Randomly generate elements $u$ of orders $3$, $9$ and $27$ in $F_4(3)$ (given by the presentation alluded to in Lemma \ref{p:almostpres}) until we hit each class. (The class to which $u$ belongs can be deduced from \cite[Tables 3 and 4]{lawther1995}.) The Jordan blocks of the action of $u$ on the sum of $L_{\mb X}(1000)$ and $L_{\mb X}(0010)$ are trivial to compute then. From \cite{lawther1995} we obtain the class in $\mb G$ to which $u$ belongs.

However, note that there is an error in \cite{lawther1995}, due to an error in \cite{mizuno1980}, which leads to a single class having the wrong Jordan blocks in characteristic $3$. This is corrected in \cite{lawther1995corr}, and it concerns exactly the class $E_8(b_6)$ in the table. It has Jordan blocks $9^{26},7,3^2,1$ on $L(E_8)$, not $9^{25},8^2,2^2,1^3$ as stated in \cite{lawther1995}. With this correction, the Jordan block structure of $u$ on $L(\mb G)$ determines the unipotent class to which $u$ belongs, and thus we are done.
\end{proof}

For completeness we do the same thing for nilpotent orbits of $\mathfrak{f}_4$. 

\begin{prop} If $x$ is a nilpotent element of the maximal $\mathfrak f_4$-subalgebra of $\mathfrak e_8$, then the class of $x$ in $\mathfrak f_4$ and $\mathfrak e_8$ is given in Table \ref{tab:classfusionliealg}.
\begin{table}
\begin{center}\begin{tabular}{cc}
\hline Class in $\mathfrak f_4$ & Class in $\mathfrak e_8$ 
\\ \hline $A_1$ & $2A_{1}$
\\ $\tilde A_1$ & $4A_{1}$
\\ $A_1+\tilde A_1$ & $A_{2}+2A_{1}$
\\ $A_2$ & $2A_{2}$
\\ $\tilde A_2$ & $A_{2}+3A_{1}$
\\ $A_2+\tilde A_1$ &  $2A_{2}+A_{1}$
\\ $B_2$ & $2A_{3}$
\\ $\tilde A_2+A_1$ & $2A_{2}+A_{1}$ 
\\ $C_3(a_1)$ & $A_{4}+2A_{1}$
\\ $F_4(a_3)$ & $A_{4}+A_{2}$
\\ $B_3$ & $A_6$
\\ $C_3$ & $D_{5}+A_{1}$
\\ $F_4(a_2)$ & $E_7(a_4)$
\\ $F_4(a_1)$ & $E_6(a_1)$ 
\\ $F_4$ & $E_6$
\\ \hline
\end{tabular}
\end{center}
\caption{Fusion of nilpotent classes of maximal $\mathfrak f_4$ into $\mathfrak e_8$.}
\label{tab:classfusionliealg}
\end{table}
\end{prop}

\begin{proof}
Using the root elements constructed for $\mathfrak f_4$ as a subalgebra of $\mathfrak e_8$ (as explained in Section \ref{s:sec3}) we find a set of nilpotent orbit representatives for $\mathfrak f_4$ using \cite[Appendix]{stewart16}. For each representative $x$ we use Magma to calculate the Jordan block structure for the adjoint action of $x$ on $\mathfrak e_8$ and the normalizer of each term of the derived series of $C_{\mathfrak e_8}(x)$. Using \cite[Proposition 1.5]{kst21un}, we then find the $\mathfrak e_8$-class of $x$. 
\end{proof}

Before stating the next result we explain some notation. We use $\bar{A}_2$ to denote an $A_2$ subgroup generated by long root elements.

In $\mb G$ there are two conjugacy classes of $B_4$ subgroups contained in $D_8$ embedded via their spin module $L_{B_4}(0001)$. When $p$ is odd, these are distinguished by whether or not they are contained in a maximal subgroup of type $A_8$. We denote the class not contained in $A_8$ by $B_4(\dagger)$.

\begin{prop}
If $\mb M$ is a maximal connected reductive subgroup of $\mb X < \mb G$, then $\mb M$ is conjugate to one of the following four subgroups. 
\begin{enumerate}[label=(\roman*)]
\item $\mb M_1 = B_4(\dagger) < D_8$ embedded via the spin module $L_{B_4}(0001)$. It is $\mb G$-irreducible and denoted $E_8(\#45)$ in \cite{thomas2017un}. 

\item $\mb M_2 = A_1 C_3 < D_8$ embedded via $L_{A_1}(2) \oplus L_{C_3}(010)$. This subgroup is $\mb G$-irreducible and denoted by $E_8(\#774)$ in \cite{thomas2017un}.

\item $\mb M_3 = A_1 G_2 < A_1 E_7$ embedded as follows: $E_7$ has a maximal subgroup of type $A_1 G_2$ (when $p\neq 2$). Therefore $A_1 E_7$ has a maximal subgroup $A_1 A_1 G_2$ and $M_3$ is embedded diagonally in this subgroup. One has to twist the embedding in the first $A_1$ factor by the Frobenius morphism. This subgroup is again $\mb G$-irreducible and denoted by $E_8(\#967^{\{1,0\}})$ in \cite{thomas2017un}. 

\item $\mb M_4 = A_2 A_2 < \bar{A}_2 E_6$ embedded as follows: $E_6$ has a maximal subgroup $A_2 G_2$, and $G_2$ has a maximal subgroup $\tilde{A}_2$ generated by short root subgroups of the $G_2$ when $p=3$. Thus $\bar{A}_2 E_6$ has a subgroup $H = \bar{A}_2 A_2 \tilde{A}_2$ (denoted $E_8(\#1012)$ in \cite{thomas2017un}). The first $A_2$ factor of $\mb M_4$ is the second $A_2$ factor of $H$ and the second $A_2$ factor of $\mb M_4$ is diagonally embedded in the first and third factors of $H$ (with no twisting by field or graph automorphisms). Moreover, $\mb M_4$ is not $\mb G$-irreducible. 
\end{enumerate}
\end{prop}

\begin{proof}
By \cite[Corollary~2]{liebeckseitz2004}, $F_4$ has four conjugacy classes of reductive maximal connected subgroups in characteristic $3$, which are indeed of types $B_4$, $A_1C_3$, $A_1G_2$ and $A_2A_2$. The first two maximal subgroups are centralizers of involutions. It follows from the action of $F_4$ on $\text{Lie}(G)$ that the centralizer in $\mb G$ of both of these involutions is $D_8$. Thus $B_4$ and $A_1C_3$ are contained in a maximal subgroup of type $D_8$. By \cite[Table~8.1]{liebeckseitz1996}, there are only three $\mb G$-conjugacy classes of $B_4$ subgroups in $D_8$. Calculating the composition factors of the action of $B_4 < F_4$ on $\text{Lie}(G)$ yields that the $B_4$ subgroup of $F_4$ is indeed conjugate to $\mb M_1$, which is $\mb G$-irreducible by \cite[Theorem~1]{thomas2017un}. Calculating the composition factors of $A_1C_3 < F_4$ on $\text{Lie}(G)$, we find that it has no trivial composition factors. Therefore it must be $\mb G$-irreducible (by \cite[Corollary~3.8]{thomas2017un}) and we may use the classification of irreducible subgroups determined in \cite{thomas2017un}. In particular, the composition factors on the Lie algebra of $\mb G$ are enough to determine conjugacy and it follows that $A_1 C_3$ is conjugate to $\mb M_2$, as required. 

Another calculation shows that $A_1G_2 < F_4$ has no trivial composition factors on the Lie algebra of $\mb G$. We can therefore use the same method as in the previous case to deduce that it is conjugate to $\mb M_3$. 

For $\mb A\mb B= A_2 A_2 < F_4$ we start by considering the first $A_2$ factor $\mb A$, which we define to be the $A_2$ subgroup generated by long root subgroups of $F_4$. This is the derived subgroup of a long root $A_2$-Levi subgroup, and is thus a subgroup of $B_4 < F_4$. Since $B_4$ is conjugate to $\mb M_1$, we find that $\mb A$ is contained in a maximal subgroup of type $A_8$ and acts as $L_{A_2}(10) \oplus L_{A_2}(01) \oplus L_{A_2}(00)^{\oplus 3}$ on the natural $9$-dimensional module of $A_8$. Given this action, it follows that $\mb A$ is a diagonal subgroup (without field or graph twists) of the derived subgroup of an $A_2A_2$-Levi subgroup of $A_8$ and hence of $\mb G$. We now claim that the connected centralizer of $\mb A$ in $\mb G$ is $\bar{A}_2 G_2$. Indeed, we use \cite[Theorem~1]{thomas2017un} to see that $\mb C = \mb A \bar{A}_2 G_2$ is $\mb G$-irreducible (denoted by $E_8(\#978)$) and the only reductive connected overgroups of $\mb C$ are the maximal subgroups $\bar{A}_2 E_6$ and $G_2 F_4$. Therefore, we must have that $\bar{A}_2 G_2 \leq C_{\mb G}(\mb A)^\circ$. Moreover, this must be an equality of subgroups since $\mb A C_{\mb G}(\mb A)^\circ$ is a reductive connected overgroup of $\mb C$, but clearly cannot be $\bar{A}_2 E_6$ or $G_2 F_4$ ($\mb A$ is not conjugate to $\bar{A}_2$). 

From the previous calculations, $\mb A\mb B < \mb C$ is a subgroup of $\bar{A}_2 E_6$ and $\mb B < C_{\mb G}(\mb A)$ is contained in $\bar{A}_2 G_2$. It is straightforward to list all $A_2$ subgroups of $\bar{A}_2 G_2$, noting that $G_2$ has precisely two classes of $A_2$ subgroups when the characteristic is $3$. Computing the composition factors of $\mb B$ on the Lie algebra of $\mb G$, by restricting first to $F_4$ and then $\mb B$, shows that $\mb B$ is conjugate to the subgroup claimed and hence $\mb A\mb B$ is conjugate to $\mb M_4$. The fact that $\mb M_4$ is $\mb G$-reducible follows from \cite[Theorem~1]{thomas2017un}. 
\end{proof}

\Addresses


\begin{thebibliography}{10}

\bibitem{Car89}
Roger W.~Carter, \emph{Simple groups of {L}ie type}, John Wiley \& Sons Inc., 1989.

\bibitem{atlas}
John Conway, Robert Curtis, Simon Norton, Richard Parker, and Robert Wilson, \emph{Atlas of finite groups}, Oxford University Press, Eynsham, 1985.

\bibitem{craven2019un}
David~A.\ Craven, \emph{On medium-rank {L}ie primitive and maximal subgroups of exceptional groups of {L}ie type}, \emph{Mem. Amer. Math. Soc.}, to appear.

\bibitem{craven2020un}
\bysame, \emph{The maximal subgroups of the exceptional groups {$F_4(q)$}, {$E_6(q)$} and {${}^2\!E_6(q)$} and related almost simple groups}, preprint, 2021. arXiv:2103.04869

\bibitem{craven2021un}
\bysame, \emph{The maximal subgroups of the exceptional groups {$E_7(q)$} and related almost simple groups}, preprint,
  2021.

\bibitem{abc}
Christoph Jansen, Klaus Lux, Richard Parker, and Robert Wilson, \emph{An {A}tlas of {B}rauer characters}, Oxford University Press, New York, 1995.


\bibitem{kst21un} Mikko Korhonen, David Stewart, and Adam Thomas, \emph{Representatives for unipotent classes and nilpotent orbits}, preprint, 2021. arXiv:2105.04347

\bibitem{lawther1995}
Ross Lawther, \emph{Jordan block sizes of unipotent elements in exceptional algebraic groups}, Comm.\ Algebra \textbf{23} (1995), 4125--4156.

\bibitem{lawther1995corr}
\bysame, \emph{Correction to: {``}{J}ordan block sizes of unipotent elements in exceptional algebraic groups{"}}, Comm.\ Algebra \textbf{26} (1998), 2709.

\bibitem{lawther2009}
\bysame, \emph{Unipotent classes in maximal subgroups of exceptional algebraic groups}, J.\ Algebra \textbf{322} (2009), 270--293.

\bibitem{liebeckseitz1996} 
Martin Liebeck and Gary Seitz, \emph{Reductive subgroups of exceptional algebraic groups}, Mem.\ Amer.\ Math.\ Soc.\ \textbf{580} (1996), no. 580, vi+111.

\bibitem{liebeckseitz1998}
\bysame, \emph{On the subgroup structure of exceptional groups of {L}ie type}, Trans.\ Amer.\ Math.\ Soc. \textbf{350} (1998), 3409--3482.

\bibitem{liebeckseitz1999}
\bysame, \emph{On finite subgroups of exceptional algebraic groups}, J.\ reine angew.\ Math.\ \textbf{515} (1999), 25--72.

\bibitem{liebeckseitz2004}
\bysame, \emph{The maximal subgroups of positive dimension in exceptional algebraic groups}, Mem.\ Amer.\ Math.\ Soc. \textbf{169} (2004).

\bibitem{litterick}
Alastair Litterick, \emph{Finite simple subgroups of exceptional algebraic groups}, Ph.D. thesis, Imperial College, London, 2013.

\bibitem{luebeck2001}
Frank L{\"u}beck, \emph{Small degree representations of finite {C}hevalley groups in defining characteristic}, LMS J.\ Comput.\ Math.\ \textbf{4} (2001), 135--169.

\bibitem{malcev1945}
Anatoly Mal'cev, \emph{Commutative subalgebras of semi-simple {L}ie algebras}, Bull. Acad. Sci. URSS Ser. Math. \textbf{9} (1945), 291--300.

\bibitem{mizuno1980}
Kenzo Mizuno, \emph{The conjugate classes of unipotent elements of the {C}hevalley groups {$E_7$} and {$E_8$}}, Tokyo J. Math. \textbf{3} (1980), 391--461.
  
\bibitem{seitz1991}Gary Seitz, \emph{Maximal subgroups of exceptional algebraic groups}, Mem.\ Amer.\ Math.\ Soc.\ \textbf{90} (1991), no. 441, iv+197. 

\bibitem{smith1976} Peter Smith, \emph{A simple subgroup of $M?$ and $E_8(3)$}, Bull.\ London Math.\ Soc.\ \textbf{8} (1976), 161--165.

\bibitem{stewart16} David Stewart, \emph{On the minimal modules for exceptional {L}ie algebras: {J}ordan blocks and stabilizers}, LMS J.\ Comput.\ Math.\ \textbf{19} (2016), 235--258.

\bibitem{SF88}Helmut Strade and Rolf Farnsteiner, \emph{Modular Lie algebras and their representations}, Monographs and Textbooks in Pure and Applied Mathematics, vol. 116, Marcel Dekker Inc., New York, 1988.

\bibitem{thomas2017un}
Adam Thomas, \emph{The irreducible subgroups of exceptional algebraic groups}, Mem. Amer. Math. Soc. 268 (2021), no. 1307, vi+191.
\end{thebibliography}
\end{document}